\documentclass[11pt,a4j,oneside]{article}
\usepackage{amscd,amsmath,amssymb,amsfonts,mathrsfs,bm}
\usepackage{amsthm}
\usepackage[abbrev]{amsrefs}
\usepackage[all]{xy}
\usepackage{color}
\usepackage{graphicx}
\usepackage{latexsym}

\numberwithin{equation}{section}

\theoremstyle{definition}
\newtheorem{thm}{Theorem}[section]
\newtheorem{prop}[thm]{Proposition}
\newtheorem{lem}[thm]{Lemma}

\newtheorem{dfn}[thm]{Definition}

\newtheorem{rem}[thm]{Remark}

\begin{document}

\title{\huge Jones Wenzl projectors in tensor products of a Verma module and irreducible modules}
\author{Ryoga Matsumoto}
\date{}%\Large 2023/12/18}
\maketitle

\begin{abstract}
We construct special idempotents in $\mathrm{End}_{U_q(\mathfrak{sl}_2)}(M(\mu)\otimes V_1^{\otimes n})$ like the Jones Wenzl projector where $M(\mu)$ is Verma module whose highest weight is $\mu$ and $V_1$ is $2$-dimensional irreducible module.
\end{abstract}

\tableofcontents

\section{Introduction}
Temperley and Lieb constructed Temperley Lieb algebra to solve the problem of statistical physics in \cites{TemperleyLieb}. Jones introduced a knot invariant called Jones polynomial using Hecke algebra in \cites{Jones}. A special idempotent in Temperley Lieb algebra called Jones Wenzl idempotent is introduced by Wenzl \cites{Wenzl}. Using the Jones Wenzl projector, Murakami and Murakami extended the Jones polynomial to knot invariants obtained from $n+1$-dimensional representations over $U_q(\mathfrak{sl}_2)$ in \cites{MurakamiMurakami}. On the other side, Iohara, Lehrer and Zhang determined the structure of $\mathrm{End}_{U_q(\mathfrak{sl}_2)}(M(\mu)\otimes V_1^{\otimes n})$ using B type Temperley Lieb algebra constructed independently by Green, Graham and Lehrer \cites{Green, GrahamLehrer} in \cites{IoharaLehrerZhang} where $M(\mu)$ is Verma module whose highest weight is $\mu$ and $V_1$ is $2$-dimensional irreducible module. However the Jones Wenzl projector like idempotent in $\mathrm{End}_{U_q(\mathfrak{sl}_2)}(M(\mu)\otimes V_1^{\otimes n})$ have not been constructed. In this article, we construct special idempotents in $\mathrm{End}_{U_q(\mathfrak{sl}_2)}(M(\mu)\otimes V_1^{\otimes n})$ like Jones Wenzl projector. More precisely, we obtain the following.
\begin{thm}
Let $cap_i$ and $cup_i$ be intertwining operators over $U_q(\mathfrak{sl}_2)$ defined by Theorem \ref{theorem}. Then there exists an element $P_{\mu,n} \in \mathrm{End}_{U_q(\mathfrak{sl}_2)}(M(\mu)\otimes V_1^{\otimes n})$ as follows.
\begin{align*}
P_{\mu,n}^2&=P_{\mu,n} \\
cap_i P_{\mu,n}&=0 \quad (i=1, \cdots, n-1) \\
P_{\mu,n}cup_i&=0 \quad (i=1, \cdots, n-1)
\end{align*}
\end{thm}

\subsection*{Acknowledgements}
We would like to thank Yuji Terashima for valuable discussions. This work was supported by JST, the establishment of university fellowships towards the creation of science technology innovation, Grant Number JPMJFS2102.

\section{Representation of $U_q(\mathfrak{sl}_2)$}
In this section, we define the quantum algebra $U_q(\mathfrak{sl}_2)$ and the representations.
\begin{dfn}
The quantum group $U_q(\mathfrak{sl}_2)$ of $\mathfrak{sl}_2 $ is the algebra on $\mathbb{C}(q)$ generated by the elements $K,K^{-1},E,F$ that satisfy the following relations.
\begin{align*}
KK^{-1}=1=K^{-1}K,\quad KE=q^2EK \quad
KF=q^{-2}FK,\quad EF-FE=\frac{K-K^{-1}}{q-q^{-1}}
\end{align*}
The coproduct $\Delta: U_q(\mathfrak{sl}_2)\rightarrow U_q(\mathfrak{sl}_2)\otimes U_q(\mathfrak{sl}_2)$ is defined below on the algebra.
\begin{align*}
\Delta(K^{\pm1}):=K^{\pm1}\otimes K^{\pm1},\quad \Delta(F):=F\otimes 1+K^{-1} \otimes F, \quad \Delta(E):=E\otimes K+1 \otimes E \\
\end{align*}
\end{dfn}

\begin{dfn}
$k+1$-dimensional irreducible representation $V_k$ has a basis $\{v_0, v_1, \cdots , v_k \}$ called an induced basis and satisfies the following relations for the generators of $U_q(\mathfrak{sl}_2)$.
\begin{align*}
K\cdot v_i:= q^{k-2i}v_i \\
E\cdot v_i:= [i]v_{i-1} \\
F\cdot v_i:= [k-i]v_{i+1}
\end{align*}
where $v_{-1}=v_{k+1}=0$. $[k]$ is defined as
\[
[k]:=\frac{q^k-q^{-k}}{q-q^{-1}}
\]
\end{dfn}

\begin{dfn}
Verma module $M(\mu)$ has a basis $\{v_0, v_1, \cdots\}$ called an induced basis and satisfies the following relations for the generators of $U_q(\mathfrak{sl}_2)$.
\begin{align*}
K\cdot v_i:= q^{\mu-2i}v_i \\
E\cdot v_i:= [i]v_{i-1} \\
F\cdot v_i:= [\mu-i]v_{i+1}
\end{align*}
where $v_{-1}=0$.
\end{dfn}

\begin{dfn}
Given a module $M, M'$ on $U_q(\mathfrak{sl}_2)$, by using the action induced from the coproduct, the tensor product representation $M \otimes M' $ is defined from the following relations. For all $m \in M $ and $m' \in M' $, we have
\begin{align*}
K^{\pm1}\cdot(m\otimes m'):=(K^{\pm1}\cdot m)\otimes (K^{\pm1}\cdot m') \\
F\cdot (m\otimes m'):=(F\cdot m)\otimes m'+(K^{-1}\cdot m)\otimes (F\cdot m') \\
E\cdot (m\otimes m'):=(E\cdot m)\otimes (K\cdot m')+m \otimes(E\cdot m')
\end{align*}
\end{dfn}

\section{Endomorphism algebras of $V_1^{\otimes n}$ and $M(\mu)\otimes V_1^{\otimes n}$}
In this section, we introduce Temperley-Lieb algebra and Temperley-Lieb algebra of type B. Hereinafter, we simply denote $v_i \otimes v_j \in M \otimes N$ by $v_{i,j}$ where $M$ and $N$ are $M(\mu)$ or $V_k$ respectively.

\begin{dfn}
The intertwining operators $cap: \mathbb{C}(q) \rightarrow V_1\otimes V_1$ and $cup: V_1\otimes V_1 \rightarrow \mathbb{C}(q)$ over $U_q\mathfrak{sl}_2$ are defined as follows.
\begin{align*}
&cap(1)=v_{0,1}-q^{-1}v_{1,0} \\
&cup(v_{0,0})=cup(v_{1,1})=0, \ cup(v_{0,1})=-q, \ cup(v_{1,0})=1
\end{align*}
\end{dfn}
\begin{dfn}
Temperley Lieb algebra $TL_n$ is a $\mathbb{C}(q)$-algebra $\mathrm{End}_{U_q(\mathfrak{sl}_2)}(V_1^{\otimes n})$. The generators of $TL_n$ is $\{e_i\}$ ($i=1, \cdots, n-1$) where $e_i$ is defined as follows.
\[
e_i= \mathrm{Id}^{\otimes (i-1)}\otimes (cup \circ cap)\otimes \mathrm{Id}^{\otimes (n-i-1)}
\]
\end{dfn}

\begin{dfn}
Temperley Lieb algebra of type B $TLB_{\mu,n}$ is a $\mathbb{C}(q,q^\mu)$ algebra $\mathrm{End}_{U_q(\mathfrak{sl}_2)}(M(\mu) \otimes V_1^{\otimes n})$.
\end{dfn}
\begin{rem}
In \cite{IoharaLehrerZhang}, it is proved that the endomorphism algebra $\mathrm{End}_{U_q(\mathfrak{sl}_2)}(M(\mu) \otimes V_1^{\otimes n})$ is isomorphic to the Temperley Lieb algebra of type B defined in \cite{GrahamLehrer}, so we denote $\mathrm{End}_{U_q(\mathfrak{sl}_2)}(M(\mu) \otimes V_1^{\otimes n})$ by $TLB_{\mu,n}$.
\end{rem}
\begin{prop}
Set $\mathbb{C}(q,q^\mu)$ linear maps $E_{\mu}: M(\mu) \otimes V_1 \rightarrow M(\mu+1)$ and $F_{\mu}: M(\mu+1) \rightarrow M(\mu) \otimes V_1$ as follows.
\begin{align*}
E_{\mu}(v_{i,0})&:= q^i v_i, \quad E_{\mu}(v_{i,1}):= v_{i+1} \\
F_{\mu}(v_i)&:= [\mu+1-i]v_{i,0}+q^{i-\mu-1}[i]v_{i-1,1}
\end{align*}
Then we have $E_{\mu} \in \mathrm{Hom}_{U_q\mathfrak{sl}_2}(M(\mu) \otimes V_1, M(\mu+1))$ and $F_{\mu} \in \mathrm{Hom}_{U_q\mathfrak{sl}_2}(M(\mu+1), M(\mu) \otimes V_1)$.
\end{prop}
\begin{proof}
We must show the commutativity $XE_{\mu}=E_{\mu}X$ and $XF_{\mu}=F_{\mu}X$ where $X=K, E, F$. For $i \geq 0$, the action below is defined as follows.
\[
K, E, F: M(\mu)\otimes V_1 \rightarrow M(\mu)\otimes V_1
\]
\begin{align*}
Kv_{i,j}&= q^{\mu+1-(i+j)}v_{i,j} \\
Ev_{i,0}&= q[i]v_{i-1,0}, \quad Ev_{i,1}= q^{-1}[i]v_{i-1,1}+v_{i,0} \\
Fv_{i,0}&= [\mu-i]v_{i+1,0}+q^{-\mu+2i}v_{i,1}, \quad Fv_{i,1}= [\mu-i]v_{i+1,1}
\end{align*}
First we will show that $F_{\mu}$ is an intertwining operator over $U_q(\mathfrak{sl}_2)$. Now we prove the following diagram is commutative.
\[
\xymatrix{
M(\mu)\otimes V_1 \ar[r]^{E_{\mu}} \ar[d]^{X} & M(\mu+1) \ar[d]^{X} \\
M(\mu)\otimes V_1 \ar[r]^{E_{\mu}} & M(\mu+1)
}
\]

Consider the case $X=K$, we obtain
\begin{align*}
&KE_{\mu,1}v_{i,0}=q^iKv_i=q^{\mu+1-i}v_i \\
&E_{\mu,1}Kv_{i,0}=q^{\mu+1-2i}E_{\mu,1}v_i=q^{\mu+1-i}v_i \\
&KE_{\mu,1}v_{i,1}=Kv_{i+1}=q^{\mu-1-2i}v_{i+1} \\
&E_{\mu,1}Kv_{i,1}=q^{\mu+1-2i-2}E_{\mu,1}v_{i,1}=q^{\mu-1-2i}v_{i+1}
\end{align*}
Then we have $KE_{\mu,1}=E_{\mu,1}K$.
Consider the case $X=E$, we obtain
\begin{align*}
EE_{\mu}v_{i,0}&= q^i Ev_i \\
&= q^i [i]v_{i-1} \\
E_{\mu}Ev_{i,0}&= q[i] E_{\mu,1}v_{i-1,0} \\
&= q^i [i]v_{i-1} \\
EE_{\mu}v_{i,1}&= Ev_{i+1} \\
&= [i+1]v_i \\
E_{\mu}Ev_{i,1}&= E_{\mu,1}(q^{-1}[i]v_{i-1,1}+v_{i,0}) \\
&= q^{-1}[i]v_i+q^i v_i \\
&= [i+1]v_i
\end{align*}
Then we have $EE_{\mu,1}=E_{\mu,1}E$.
Consider the case $X=F$, we obtain
\begin{align*}
FE_{\mu}v_{i,0}&= q^i Fv_i \\
&= q^i [\mu+1-i]v_{i+1} \\
E_{\mu}Fv_{i,0}&= E_{\mu,1}([\mu-i]v_{i+1,0}+q^{-\mu+2i}v_{i,1}) \\
&= q^{i+1}[\mu-i]v_{i+1}+q^{-\mu+2i}v_{i+1} \\
&= q^i [\mu+1-i]v_{i+1} \\
FE_{\mu}v_{i,1}&= Fv_{i+1} \\
&= [\mu-i]v_{i+2} \\
E_{\mu}Fv_{i,1}&= [\mu-i]E_{\mu}v_{i+1,1} \\
&= [\mu-i]v_{i+2}
\end{align*}
Then we have $FE_{\mu,1}=E_{\mu,1}F$. From the computations above, $E_{\mu}$ is an intertwining operator over $U_q(\mathfrak{sl}_2)$. Next we will show that $F_{\mu}$ is an intertwining operator over $U_q(\mathfrak{sl}_2)$. Now we prove the following diagram is commutative.
\[
\xymatrix{
M(\mu+1) \ar[r]^{F_{\mu}} \ar[d]^{X} & M(\mu)\otimes V_1 \ar[d]^{X} \\
M(\mu+1) \ar[r]^{F_{\mu}} & M(\mu)\otimes V_1
}
\]
Consider the case $X=K$, we obtain
\begin{align*}
KF_{\mu}v_i&= K([\mu+1-i]v_{i,0}+q^{i-\mu-1}[i]v_{i-1,1}) \\
&= q^{\mu+1-2i}([\mu+1-i]v_{i,0}+q^{i-\mu-1}[i]v_{i-1,1}) \\
F_{\mu}Kv_i&= q^{\mu+1-2i}F_{\mu,1}v_i \\
&= q^{\mu+1-2i}([\mu+1-i]v_{i,0}+q^{i-\mu-1}[i]v_{i-1,1}) \\
\end{align*}
Then we have $KE_{\mu,1}=E_{\mu,1}K$. Consider the case $X=E$, we obtain
\begin{align*}
EF_{\mu}v_i&= E([\mu+1-i]v_{i,0}+q^{i-\mu-1}[i]v_{i-1,1}) \\
&= q[i][\mu+1-i]v_{i-1,0}+q^{i-\mu-1}[i](q^{-1}[i-1]v_{i-2,1}+v_{i-1,0}) \\
&= [i](q[\mu+1-i]+q^{i-\mu-1})v_{i-1,0}+q^{i-\mu-2}[i][i-1]v_{i-2,1} \\
&= [i][\mu+2-i]v_{i-1,0}+q^{i-\mu-2}[i][i-1]v_{i-2,1} \\
F_{\mu}Ev_i&= [i]F_{\mu}v_{i-1} \\
&= [i]([\mu+2-i]v_{i-1,0}+q^{i-\mu-2}[i-1]v_{i-2,1})
\end{align*}
Then we have $EE_{\mu,1}=E_{\mu,1}E$. Consider the case $X=F$, we obtain
\begin{align*}
FF_{\mu}v_i&= F([\mu+1-i]v_{i,0}+q^{i-\mu-1}[i]v_{i-1,1}) \\
&= [\mu+1-i][\mu-i]v_{i+1,0}+q^{-\mu+2i}[\mu+1-i]v_{i,1}+q^{i-\mu-1}[i][\mu+1-i]v_{i,1} \\
&= [\mu+1-i][\mu-i]v_{i+1,0}+q^{i-\mu}[\mu+1-i](q^i+q^{-1}[i])v_{i,1} \\
&= [\mu+1-i]([\mu-i]v_{i+1,0}+q^{i-\mu}[i+1]v_{i,1}) \\
F_{\mu}Fv_i&= [\mu+1-i]F_{\mu,1}v_{i+1} \\
&= [\mu+1-i]([\mu-i]v_{i+1,0}+q^{i-\mu}[i+1]v_{i,1})
\end{align*}
Then we have $FE_{\mu,1}=E_{\mu,1}F$. From the results above, $F_{\mu}$ is an intertwining operator.
\end{proof}

\section{Main Theorem}
In this section, we define special idempotents in $TLB_{\mu,n}$ like the Jones Wenzl projector in $TL_n$.
\begin{dfn}
Jones Wenzl projector $P'_n \in TL_n$ is defined as follows.
\begin{align*}
P'_1:=\mathrm{Id}, \quad P'_n:=P'_{n-1}+\frac{[n-1]}{[n]}P'_{n-1}e_{n-1}P'_{n-1}
\end{align*}
\end{dfn}

\begin{prop}
We have
\begin{align*}
(P'_n)^2&=P'_n \\
e_i P'_n&=P'_n e_i=0 \quad (i=1,\ldots,n-1)
\end{align*}
\end{prop}
\begin{proof}
See Proposition 2 in \cites{KauffmanLins}.
\end{proof}

\begin{dfn}
Put $E_{\mu,i}:= E_\mu \otimes \mathrm{Id}^{\otimes i}$ and $F_{\mu,i}:= F_\mu \otimes \mathrm{Id}^{\otimes i}$. Extended Jones Wenzl projector $P_{\mu,n} \in TLB_{\mu,n}$ is defined as follows.
\begin{align*}
P_{\mu,1}:= \frac{F_{\mu}E_{\mu}}{[\mu+1]}, \quad P_{\mu,n}:= \frac{F_{\mu,n-1}F_{\mu+1,n-2}\cdots F_{\mu+n-2,1}F_{\mu+n-1} E_{\mu+n-1} E_{\mu+n-2,1}\cdots E_{\mu+1,n-2}E_{\mu,n-1}}{[\mu+n][\mu+n-1]\cdots[\mu+1]}
\end{align*}
\end{dfn}
\begin{rem}
The definition of Extended Jones Wenzl projectors is inspired by Definition 2.11 in \cites{RoseTubbenhauer}. By Corollary 2.13 in \cites{RoseTubbenhauer}, the Jones Wenzl projectors in \cites{RoseTubbenhauer} coincide with $P_n'$.
\end{rem}

\begin{lem} \label{EF=[mu+1]}
We have
\[
E_\mu F_\mu= [\mu+1]\mathrm{Id}
\]
\end{lem}
\begin{proof}
It is given by directly computing $E_\mu F_\mu$.
\begin{align*}
E_\mu F_\mu v_{i}&=[\mu+1-i]v_{i,0}+q^{i-\mu-1}[i]v_{i-1,1} \\
&= [\mu+1-i]q^i v_i+ q^{i-\mu-1}[i]v_i \\
&=\frac{q^{\mu+1}-q^{-\mu-1+2i}+q^{2i-\mu-1}-q^{-\mu-1}}{q-q^{-1}}v_i \\
&= [\mu+1]v_i
\end{align*}
\end{proof}

\begin{lem} \label{capP=0}
We have
\begin{align*}
(\mathrm{Id}_\mu \otimes cap)(F_\mu \otimes \mathrm{Id})F_{\mu+1}=0 \\
E_{\mu+1}(E_\mu \otimes \mathrm{Id})(\mathrm{Id}_\mu \otimes cup)=0
\end{align*}
\end{lem}
\begin{proof}
It is directly obtained using the definition of $cap$, $cup$, $E_\mu$ and $F_\mu$.
\begin{align*}
&(\mathrm{Id}_\mu \otimes cap)(F_\mu \otimes \mathrm{Id})F_{\mu+1}v_i \\
&= (\mathrm{Id}_\mu \otimes cap)(F_\mu \otimes \mathrm{Id})([\mu+2-i]v_{i,0}+q^{i-\mu-2}[i]v_{i-1,1}) \\
&= (\mathrm{Id}_\mu \otimes cap)([\mu+2-i]([\mu+1-i]v_{i,0,0}+q^{i-\mu-1}[i]v_{i-1,1,0}) \\
&\qquad +q^{i-\mu-2}[i]([\mu+2-i]v_{i-1,0,1}+q^{i-\mu-2}[i-1]v_{i-2,1,1})) \\
&= q^{i-\mu-1}[i][\mu+2-i]v_{i-1}-q^{i-\mu-1}[i][\mu+2-i]v_{i-1} \\
&= 0
\end{align*}

\begin{align*}
&E_{\mu+1}(E_\mu \otimes \mathrm{Id})(\mathrm{Id}_\mu \otimes cup)v_i \\
&= E_{\mu+1}(E_\mu \otimes \mathrm{Id})(v_{i,0,1}-q^{-1}v_{i,1,0}) \\
&= E_{\mu+1}(q^i v_{i,1}-q^{-1}v_{i+1,0}) \\
&= q^i v_{i+1}-q^i v_{i+1} \\
&= 0
\end{align*}
\end{proof}

\begin{thm} \label{theorem}
Set $cap_i:= \mathrm{Id}_\mu \otimes \mathrm{Id}^{\otimes i} \otimes cap \otimes \mathrm{Id}^{\otimes (n-i-1)}$ and $cup_i:= \mathrm{Id}_\mu \otimes \mathrm{Id}^{\otimes i} \otimes cup \otimes \mathrm{Id}^{\otimes (n-i-1)}$. Extended Jones Wenzl projectors $P_{\mu,n} \in TLB_{\mu,n}$ satisfy the following conditions.
\begin{align*}
P_{\mu,n}^2&=P_{\mu,n} \\
cap_i P_{\mu,n}&=0 \quad (i=1, \cdots, n-1) \\
P_{\mu,n} cup_i&=0 \quad (i=1, \cdots, n-1)
\end{align*}
\end{thm}
\begin{proof}
First we show $P_{\mu,n}^2=P_{\mu,n}$. From Lemma \ref{EF=[mu+1]},
\begin{align*}
&E_{\mu+n-1}E_{\mu+n-2,1}\cdots E_{\mu,n-1}F_{\mu,n-1}\cdots F_{\mu+n-2,1}F_{\mu+n-1} \\
&= [\mu+1]E_{\mu+n-1}E_{\mu+n-2,1}\cdots E_{\mu+1,n-2}F_{\mu+1,n-2}\cdots F_{\mu+n-2,1}F_{\mu+n-1} \\
&= \cdots \\
&= [\mu+n][\mu+n-1]\cdots [\mu+1]\mathrm{Id}_{\mu+n-1}
\end{align*}
Then we obtain
\begin{align*}
P_{\mu,n}^2&= \frac{1}{([\mu+n]\cdots [\mu+1])^2}F_{\mu,n-1}\cdots F_{\mu+n-1} E_{\mu+n-1}\cdots E_{\mu,n-1}F_{\mu,n-1} \\
&\qquad \cdots F_{\mu+n-1} E_{\mu+n-1}\cdots E_{\mu,n-1} \\
&= \frac{F_{\mu,n-1}\cdots F_{\mu+n-1}E_{\mu+n-1}\cdots E_{\mu,n-1}}{[\mu+n]\cdots [\mu+1]} \\
&= P_{\mu,n}
\end{align*}
Next, we prove the second formula. By Lemma \ref{capP=0}, we have
\begin{align*}
&cap_i P_{\mu,n} \\
&= \frac{1}{[\mu+n]\cdots [\mu+1]}F_{\mu,n-1}\cdots (F_{\mu+i-2} \otimes \mathrm{Id}^{\otimes (n-i+1)})(\mathrm{Id}_{\mu+i-1}\otimes cap \otimes \mathrm{Id}^{\otimes (n-i-1)}) \\
&\qquad (F_{\mu+i-1} \otimes \mathrm{Id}^{\otimes (n-i)})(F_{\mu+i} \otimes \mathrm{Id}^{\otimes (n-i-1)})\cdots F_{\mu+n-1}E_{\mu+n-1}\cdots E_{\mu,n-1} \\
&= 0
\end{align*}
Finally, we show the third formula. By Lemma \ref{capP=0}, we get
\begin{align*}
&P_{\mu,n}cup_i \\
&= \frac{1}{[\mu+n]\cdots [\mu+1]}F_{\mu,n-1}\cdots F_{\mu+n-1}E_{\mu+n-1}\cdots (E_{\mu+i}\otimes \mathrm{Id}^{\otimes (n-i-1)})(E_{\mu+i-1}\otimes \mathrm{Id}^{\otimes (n-i)}) \\
&\qquad (\mathrm{Id}_{\mu+i-1}\otimes cup \otimes \mathrm{Id}^{n-i-1})(E_{\mu+i-2}\otimes \mathrm{Id}^{\otimes (n-i+1)})\cdots E_{\mu,n-1} \\
&= 0
\end{align*}
\end{proof}

\footnotesize{DEPARTMENT OF MATHEMATICS, TOHOKU UNIVERSITY, 6-3, AOBA, ARAMAKI-AZA, AOBA-KU, SENDAI, 980-8578, JAPAN} \par
\footnotesize{\textit{Email Address}}: \texttt{matsumoto.ryoga.t2@dc.tohoku.ac.jp}

\end{document}